\documentclass[11pt]{amsart}

\usepackage{enumerate}
\usepackage[margin =1in]{geometry}
\usepackage{amsmath}
\usepackage{graphicx}
\usepackage{comment}
\usepackage{wrapfig}
\usepackage{caption}
\usepackage{mathtools}
\usepackage{amssymb}
\usepackage{multirow}
\usepackage{tikz}
\usetikzlibrary{trees}
\specialcomment{itsapicture}{}{}
\specialcomment{nomathabstract}{}{}
\excludecomment{nomathabstract}

\newcommand\Tstrut{\rule{0pt}{2.6ex}}         
\newcommand\Bstrut{\rule[-1ex]{0pt}{0pt}}   
\newcommand{\ds}{\displaystyle}
\newcommand{\inv}{^{-1}}
\newcommand{\floor}[1]{\lfloor #1 \rfloor}
\newcommand{\Floor}[1]{\left\lfloor #1 \right\rfloor}

\newcommand{\cal}{\mathcal}
\newcommand{\bbf}{\mathbb{F}}

\newcommand{\bbq}{\mathbb{Q}}

\newcommand{\bbz}{\mathbb{Z}}

\newcommand{\disc}{\operatorname{disc}}

\newcommand{\ind}{\operatorname{ind}}

\newcommand{\red}{\operatorname{red}}
\newcommand{\tr}{\operatorname{tr}}

\newcommand{\mmod}[1]{\,(\!\!\!\!\mod{#1})}

\newtheorem{thm}{Theorem}[section]
\newtheorem{prop}[thm]{Proposition}
\newtheorem{cor}[thm]{Corollary}
\newtheorem{lem}[thm]{Lemma}

\theoremstyle{definition}
\newtheorem{conj}[thm]{Conjecture}
\newtheorem{defn}[thm]{Definition}

\newtheorem{example}[thm]{Example}

\theoremstyle{remark}
\newcounter{remarks}
\newcounter{notations}
\newtheorem{remark}[remarks]{Remark}
\newtheorem{notation}[notations]{Notation}

\allowdisplaybreaks

\begin{document}
\title{Discriminants of Chebyshev Radical Extensions}
\author[thomas alden gassert]{Thomas Alden Gassert}
\email{gassert@math.umass.edu}
\address{Department of Mathematics and Statistics, University of Massachusetts, Amherst, 710 N. Pleasant Street, Amherst, MA, USA 01003}

\date{\today}

\begin{abstract}
Let $t$ be any integer and fix an odd prime $\ell$. Let $\Phi(x) = T_\ell^n(x)-t$ denote the $n$-fold composition of the Chebyshev polynomial of degree $\ell$ shifted by $t$. If this polynomial is irreducible, let $K = \bbq(\theta)$, where $\theta$ is a root of $\Phi$. A theorem of Dedekind's gives a condition on $t$ for which $K$ is monogenic. For other values of $t$, we apply the Montes algorithm to obtain a formula for the discriminant of $K$ and to compute basis elements for the ring of integers ${\cal O}_K$.
\end{abstract}

\maketitle

\section{Introduction}
Let $k = \bbq(\theta)$ be a number field where $\theta$ is the root of a monic, irreducible polynomial $f(x) \in \bbz[x]$. A classical problem in number theory is the determination of the discriminant of such a number field $k$, which is closely related to the discriminant of the polynomial $f(x)$ (see equations \eqref{eq:1.1}, \eqref{eq:1.2}, \eqref{eq:1.3}). In this paper we focus on number fields that arise from iterating a particular family of polynomials, namely the Chebyshev polynomials (of the first kind), which we define in section \ref{sec:prelim}. We use the standard notation from dynamics to denote the iterates of a polynomial:

\begin{notation}
Let $f \in \bbq[x]$ be a polynomial of degree at least 2. Define $f^n(x) = f(f^{n-1}(x))$ to be the $n$-fold iterate of $f(x)$ under composition with $f^0(x) = x$. An exponent after the argument will be used to denote the $n$-th product, i.e. $f(x)^n := \big(f(x)\big)^n$.
\end{notation}

The polynomials of interest are the iterates $T_\ell^n(x)-t$ where $T_d(x)$ is the Chebyshev polynomial of degree $d$, $\ell$ is an odd prime, and $t$ is a fixed integer for which every iterate is irreducible. It is known that for every $\ell$ there are infinitely many values $t$ for which the iterates $T_\ell^n(x)-t$ are all irreducible \cite{Gassert12}. For example, when $\nu_\ell(t)=1$, the polynomials are Eisenstein at $\ell$ (Lemma \ref{lem:chebymodl}). A root $\theta_n$ of $T_\ell^n(x)-t$ is what we call a \emph{Chebyshev radical}, and we call the number field $\bbq(\theta_n)$ a \emph{Chebyshev radial extension}. We remind the reader of the standard discriminant formulas.

\begin{notation}
Let $k = \bbq(\theta)$ be a number field where $\theta$ is the root of a monic, irreducible polynomial $f(x) \in \bbz[x]$, as originally defined. We write $D(f)$ for the discriminant of the polynomial and $\Delta(k)$ for the discriminant of the number field. These discriminants are given by
\begin{align} \label{eq:1.1}
D(f) = \prod_{1\le i < j \le d} (\theta_j-\theta_i)^2
\end{align}
where $f$ has roots $\theta_1, \ldots, \theta_n$, and
\begin{align} \label{eq:1.2}
\Delta(k) = \det(\tr_{k/\bbq}(\alpha_i\alpha_j))
\end{align}
where $\alpha_1, \ldots, \alpha_d$ is a basis for the ring of integers $\cal O_k$. 
\end{notation}

The discriminant provides, in some sense, a measure of the arithmetic complexity of the underlying ring: $\bbz[\theta]$ in the case of $D(f)$, and $\cal O_k$ in the case of $\Delta(k)$. Furthermore, $\bbz[\theta] \subset \cal O_k$, and the discriminant scales as a square relative to the index $[\cal O_k \colon \bbz[\theta]]$, for which we write
\begin{align} \label{eq:1.3}
[{\cal O}_k \colon \bbz[\theta]]^2 = \frac{D(f)}{\Delta(k)}.
\end{align}

\begin{notation} For simplicity, we will write $\Phi(x) := T_\ell^n(x)-t$ where $\ell$, $n$, and $t$ are understood to satisfy the properties listed above. Unless otherwise stated, we use $\theta$ to denote a root of $\Phi$, and we write $K:=\bbq(\theta)$. Lastly, we write $\ind(\Phi) := [\cal O_K \colon \bbz[\theta]]$.
\end{notation}

In this paper we work towards an alternative formula for $\Delta(K)$. A simple formula for $D(\Phi)$ depending only on $\ell$ and $t$ is known (Proposition \ref{prop:disc}), leaving the majority of this paper to determine $\ind(\Phi)$. In section \ref{sec:monogenic}, we use Dedekind's criterion to identify the exact conditions under which a prime $p$ divides $\ind(\Phi)$ (Theorem \ref{th:monogenic}). In particular, these conditions precisely identify the values $t$ for which the number field $K$ is monogenic, that is $\cal O_K = \bbz[\theta]$. The primes that divide $\ind(\Phi)$ fall into two categories: 
\begin{enumerate}
\item $\ell \mid \ind(\Phi)$ if and only if $T_\ell(t) \equiv t \mmod{\ell^2}$
\item $p \mid \ind(\Phi)$ if and only if $t \equiv \pm2 \mmod{p^2}$. 
\end{enumerate}
In section \ref{sec:gmn method}, we introduce the Montes algorithm \cite{GMN08,GMN09,GMN12} for computing $\ind(f)$. In section \ref{sec:ell-index}, we apply the algorithm to the case where $\Phi(t) \equiv 0 \mmod{\ell^2}$ but $t \not\equiv \pm2 \mmod{\ell^2}$ and obtain a closed formula for $\Delta(K)$ when $t\pm2$ are square-free (Corollary \ref{cor:fielddisctpm2sf}). Additionally, in section \ref{sec:integral basis}, we determine generators for the ring of integers $\cal O_K$ when $t$ is subjected to the same constraints (Theorem \ref{th:ellintegralbasis}). In section \ref{sec:p-index}, we determine $\nu_p(\ind(\Phi))$ when $t \equiv \pm2 \mmod{p^2}$ (Theorem \ref{th:6.1}) and give an improved formula for $\Delta(K)$ (Corollary \ref{cor:6.3}). Most of our results are accompanied by examples.

\section{Preliminaries: Properties of Chebyshev polynomials} \label{sec:prelim}
We begin by recalling, without proof, some of the properties of Chebyshev polynomials (of the first kind) $T_d(x)$ and (of the second kind) $U_d(x)$ (e.g. see Rivlin \cite{Rivlin90} or Silverman \cite{Silverman07}).
\begin{enumerate}
\item For each integer $d \ge 0$, $T_d(x) \in \bbq[x]$ is the unique monic polynomial satisfying
$$T_d(z + z\inv) = z^d + z^{-d}$$
in the field $\bbq(z)$. Moreover, $T_d(x)$ is a degree $d$ polynomial with integral coefficients.

\item For each integer $d \ge 0$, 
$$\ds U_d(x) = \frac{d}{dx} \frac{T_{d+1}(x)}{d+1}$$ 
is a monic, integral polynomial of degree $d$.

\item $T_d(T_e(x)) = T_{de}(x)$ for all $d,e \ge 0$.

\item $T_d(-x) = (-1)^dT_d(x)$, \quad $U_d(-x) = (-1)^dU_d(x)$.

\item For all $d \ge 0$, the Chebyshev polynomials satisfy the recurrence relation
\begin{align*}
&T_{d+2}(x) = xT_{d+1}(x) - T_d(x), &&U_{d+2}(x) = xU_{d+1}(x) - U_d(x).
\end{align*}

\item For all $d \ge 0$, the Chebyshev polynomials satisfy the trigonometric relations
\begin{align*}
&T_d(2 \cos(\theta)) = 2\cos(d\theta), &&U_d(2\cos(\theta)) = \frac{\sin((d+1)\theta)}{\sin(\theta)}.
\end{align*}
\item For all $d \ge 1$, the Chebyshev polynomials are given by the explicit formulas
\begin{align*}
T_d(x) &= \sum_{k=0}^{\floor{d/2}}(-1)^k\frac{d}{d-k}{d-k \choose k}x^{d-2k}, &&U_d(x) = \sum_{k=0}^{\floor{d/2}}(-1)^k {d-k \choose k}x^{d-2k}.
\end{align*}

\item Equivalently, 
\begin{align*}
U_d(x) &= \frac{\left(x+\sqrt{x^2-4}\right)^{d+1} - \left(x-\sqrt{x^2-4}\right)^{d+1}}{2^{d+1}\sqrt{x^2-4}} \quad \text{ if } \quad x \neq \pm 2.
\end{align*}
\end{enumerate}

We proceed with some results regarding the factorization of Chebyshev polynomials.

\begin{lem}\label{lem:2.1}
If $d$ is an odd integer and $t = \pm2$, then 
\begin{align*}
T_d(x)-t = (x-t)\tau(x)^2
\end{align*}
for some monic polynomial $\tau(x) \in \bbz[x]$ of degree $(d-1)/2$. Moreover, if $d=\ell^n$ for an odd prime $\ell$, then $T_\ell^n(x)-t$ factors into irreducibles as
\begin{align*}
T_\ell^n(x)-t = (x-t)\phi_1(x)^2 \cdots \phi_r(x)^2
\end{align*}
where $\phi_i(x)$ has degree $(\ell^i-\ell^{i-1})/2$ for $i = 1,\ldots,r$.
\end{lem}

\begin{proof}
Suppose $t = 2$. Recall that the Chebyshev polynomials satisfy the trigonometric relations 
\begin{align*}
T_d(2\cos(\theta)) - 2= 2\cos(d\theta) - 2 \quad \text{ and } \quad U_{d-1}(2 \cos(\theta)) = \frac{\sin (d\theta)}{\sin(\theta)},
\end{align*} 
and note that 
$$\frac{d}{dx} \Big(T_d(x) - 2\Big) = dU_{d-1}(x).$$
Certainly, 2 is a root of $T_d(x)-2$, and 
$$\theta_i = 2\cos\left(\frac{2i\pi}{d}\right); \quad i = 1, \ldots, \frac{d-1}{2}$$ 
are roots of $T_d(x)-2$ and its derivative. It follows that
$$T_d(x)-2 = (x-2)\prod_{i=1}^{(d-1)/2}(x-\theta_i)^2.$$
A similar argument applies in the case $t = -2$.

The factorization of $T_\ell^n(x)-t$ follows from the fact that the splitting field of $T_\ell^n(x)-t$ is $\bbq(\zeta_{\ell^n})^+$, the maximal totally-real subfield of $\bbq(\zeta_{\ell^n})$, where $\zeta_{\ell^n}$ is a primitive $\ell^n$th root of unity.
\end{proof}

Recall our previously defined notation $\Phi(x)=T_\ell^n(x)-t$. 

\begin{lem} \label{lem:chebymodl}
We have $\Phi(x) \equiv (x-t)^{\ell^n} \mmod \ell$.
\end{lem}

\begin{proof} 
Recall that 
$$T_\ell(x) = \sum_{k=0}^{\floor{ \ell/2}}(-1)^k\frac{(\ell-k-1)!}{k!(\ell-2k)!}\ell x^{\ell-2k}.$$
Note that
$$\nu_\ell\left(\frac{(\ell-k-1)!}{k!(\ell-2k)!}\ell\right) = 
\begin{cases}
0 & \mbox{ if } k=0\\
1 & \mbox{ otherwise,}\\
\end{cases}$$
where $\nu_\ell$ is the standard $\ell$-adic valuation, and thus $T_\ell(x) \equiv x^\ell \mmod \ell$. It follows that
$$T_\ell^n(x)-t \equiv x^{\ell^n}-t \equiv (x-t)^{\ell^n}\pmod \ell.$$
\end{proof}

\begin{notation} We use $\overline{\phantom l}$ to denote reduction modulo a prime $p$.
\end{notation}

\begin{prop} \label{prop:2.3}
Let $p$ be an odd prime different from $\ell$ such that $t \equiv \pm 2 \mmod{p^2}$. Let $\mu$ be the least positive integer for which $\nu_\ell(p^{2\mu}-1)\ge 1$, and define $h = \nu_\ell(p^{2\mu}-1)$.

If $1 \le n < h$, then $\overline\Phi(x)$ has $(\ell^n-1)/2\mu$ distinct irreducible factors of degree $\mu$. That is,
\begin{align*}
\Phi(x) \equiv (x-\overline t) \prod_{i=1}^{\frac{\ell^n-1}{2\mu}}\phi_i(x)^2 \pmod p,
\end{align*}
where the $\phi_i$'s are distinct irreducible factors of degree $\mu$.

Otherwise, if $n = h+e \ge h$, $\overline\Phi(x)$ has $(\ell^h-1)/2\mu$ distinct irreducible factors of degree $\mu$, and $\ell^{h-1}$ distinct irreducible factors of degree $\ell^j\mu$ for each integer $1 \le k \le e$. More precisely,
\begin{align*}
\Phi(x) \equiv (x-\overline t)\prod_{i=1}^{\frac{\ell^{h-1}-1}{2\mu}}\phi_i(x)^2\prod_{j=1}^{\frac{\ell^{h-1}(\ell-1)}{2\mu}}\prod_{k=0}^e\psi_j(T_\ell^k(x))^2 \pmod p,
\end{align*}
where the $\psi_j$'s are irreducible factors of $\overline{T_\ell^h(x)- t}$ of degree $\mu$, distinct from the $\phi_i$'s.
\end{prop}

\begin{remark}
We will need to set up the tools for the proof of this proposition. The heuristic for the proof is the following: to every root $\theta$ of $\overline\Phi$ we assign a weight $w = [\bbf_p(\theta):\bbf_p]$. If we can determine the weights of all the roots of $\overline\Phi$, then we know the degrees of the irreducible factors of $\overline\Phi$. The results from \cite{Gassert12} describe the weights completely. We provide some terminology to understand the results in that paper.
\end{remark}

The action of the Chebyshev polynomial $T_\ell(x)$ (or any other polynomial) on a finite field $\bbf_p$ and its extensions can be realized in the form of a directed graph. Each value in the field corresponds to a vertex in the graph, and the graph contains a directed edge from $\beta$ to $\alpha$ if and only if $T_\ell(\beta) = \alpha$. The preimages of any value $\alpha$ can be found by tracing backwards along the paths terminating at $\alpha$.

\begin{defn}
The \emph{backwards orbit} of $\alpha$ is the set of all preimages of $\alpha$ under iteration by $T_\ell(x)$. In general, we restrict our discussion to the preimages contained within a certain finite field, and we write
$$\overleftarrow{O}_{\alpha}(\bbf_{p^m}) :=\{\theta \in \bbf_{p^m} \colon T_\ell^n(\theta) = \alpha, n \ge 1\}.$$
\end{defn}

The elements in the backward orbit of $\alpha$ can be arranged into tree graphs attached to $\alpha$. We use the following terminology to describe these trees.

\begin{defn}
The \emph{root} $r$ of a directed tree graph is a specialized point in the graph towards which all edges are directed. The \emph{height} of any vertex $v$ is the length of the (unique) path from $v$ to $r$. If the graph is finite, then the height of the graph is the length of the longest path contained in the graph. The vertices of a tree that have no incoming edges are often called \emph{leaves}. A tree is called \emph{$n$-ary} if every vertex that is not a leaf has $n$ incoming edges. We say that an $n$-ary tree is \emph{complete} $n$-ary if the height of every leaf is equal to the height of the graph.
\end{defn}

\begin{thm}[\cite{Gassert12}, Theorem 2.6] \label{th:trees}
If $\ell$ is an odd prime and $\nu_\ell(p^{2m}-1) \ge 1$, then $-2$ and $2$ are fixed and attached to each of these vertices are $(\ell-1)/2$ complete $\ell$-ary trees of height $\nu_\ell(p^{2m}-1)-1$.
\end{thm}

\begin{lem}[\cite{Gassert12}, Lemma 3.2] \label{lem:trees}
Let $\mu$ be the least positive integer for which $\nu_\ell(p^{2\mu}-1)\ge 1$. If $\mu \mid m$, then
$$\nu_\ell(p^{2m}-1) = \nu_\ell(p^{2\mu}-1) + \nu_\ell(m).$$
\end{lem}

We return to the proof of Proposition \ref{prop:2.3}.

\begin{proof} (Proposition \ref{prop:2.3})
Recall that we defined $\mu$ to be the least positive integer for which $\nu_\ell(p^{2\mu}-1) \ge 1$, and we let $h = \nu_\ell(p^{2\mu}-1)$. Therefore, Theorem \ref{th:trees} implies that $\bbf_{p^\mu}$ is the minimal extension of $\bbf_p$ containing preimages of $\overline t$. A complete $\ell$-ary tree of height $h-1$ contains $1+\ell+\cdots+ \ell^{h-1} = \frac{\ell^h-1}{\ell-1}$ points, and thus 
$$\#\overleftarrow{O}_{\overline t}(\bbf_{p^\mu}) = \frac{\ell-1}{2}\cdot \frac{\ell^h-1}{\ell-1} = \frac{\ell^h-1}{2}.$$
Each of these values is a root of an irreducible polynomial of degree $\mu$, hence for $1 \le n \le h$, the degree of each of the irreducible factors of $\overline{\frac{T_\ell^n(x)-t}{x-t}}$ is $\mu$. By Lemma \ref{lem:trees}, the smallest field containing all the roots of $\overline{T_\ell^{h+e}(x)-t}$ is $\bbf_{p^{\mu\ell^e}}$. It follows by induction that $\overline{\frac{T_\ell^{h+e}(x)-t}{x-\overline t}}$ has $\frac{\ell^h-\ell^{h-1}}{2\mu}$ factors of degree $\mu\ell^k$ for $1\le k \le e$ and $\frac{\ell^h-1}{2\mu}$ factors of degree $\mu$.
\end{proof}

\begin{example}
This example is meant as a illustrative description of the previous results. Consider the action of $T_5(x)$ on $\overline \bbf_7$, i.e. $\ell=5$ and $p=7$. One can verify that for this choice of primes, $\mu=2$, $h = 2$, and $T_5(x) = (x \pm2) (x^2\mp x-1)^2$ By Theorem \ref{th:trees}, each of the two solutions to $x^2\pm x-1$ is the root of a complete $5$-ary tree, and these roots are elements of $\bbf_{7^2}$. Moreover, this theorem and the associated lemma imply that every vertex at height $k\ge1$ has weight $\mu\ell^{k-1}$. See Figure \ref{fig:tree_fig}.
\end{example}

\begin{itsapicture}
\begin{figure}[h!]
\centering
\begin{tikzpicture}[>=stealth,every node/.style={circle}]
\draw (0,-1) node[red] (origin) {\Huge$\bullet$}
	node[right] {$\pm2$};
\draw (origin) edge[loop below] (origin);
\foreach \x in {0,1} {
	\node[brown] (\x) at (180*\x:1) {$\blacksquare$};
	\draw[->>] (\x) -- (origin);
	\foreach \y in {2,...,6} {
		\node[brown] (\y) at (180*\x+180/5*\y-180/5*4:3) {$\blacksquare$};
		\draw[->] (\y) -- (\x);
		\foreach \z in {7,...,11} {
			\node[blue] (\z) at (180*\x+180/5*\y-180/5*4+180/25*\z-180/25*9:5) {$\blacktriangle$};
			\draw[->] (\z) -- (\y);
			\foreach \w in {12,...,16} {
				\node[green] (\w) at (180*\x+180/5*\y-180/5*4+180/25*\z-180/25*9+180/125*\w-180/125*14:7.5) {$\bullet$};
				\draw (\w) -- (\z);
				}
			}
		}
	}
\end{tikzpicture}
\caption{The roots of $T_5^4(x)-\overline t$ in $\overline\bbf_7$. A double arrow corresponds to the multiplicity of the preimage. The color and shape of the vertex corresponds to its weight: 1 ({\Huge{\color{red}$\bullet$}}); 2 ({\color{brown}$\blacksquare$}); 10 ({\color{blue}$\blacktriangle$}); 50 ({\color{green}$\bullet$}).} \label{fig:tree_fig}
\end{figure}
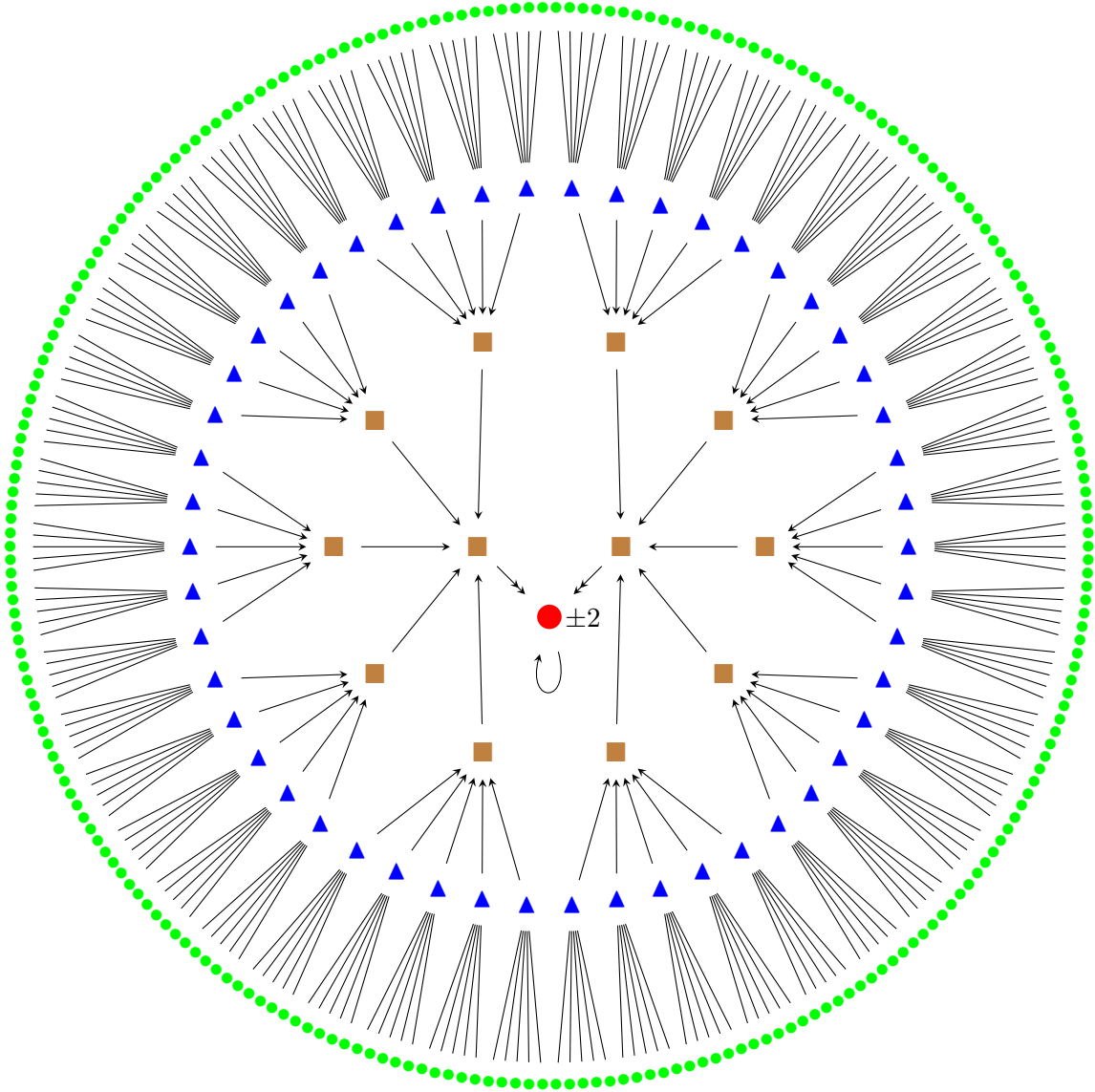
\end{itsapicture}

Lastly we provide a formula for the discriminant of $\Phi(x) = T_\ell^n(x)-t$. A general formula for the discriminant of an iterated polynomial is given by Aitken, Hajir, and Maire \cite{AHM05}. The following formula is a direct consequence of their result.

\begin{prop}[\cite{Gassert12}, Corollary 3.6.] \label{prop:disc} 
We have $$D(\Phi) = \ell^{n\ell^n}(4-t^2)^{(\ell^n-1)/2}.$$
\end{prop}

\section{Monogenic number fields} \label{sec:monogenic}
In this section we identify sufficient conditions on $t$ for which $\ind(\Phi) = 1$, or equivalently $D(\Phi) = \Delta(K)$. In particular, these are sufficient conditions for $K$ to be a monogenic number field, meaning that the ring of integers $\cal O_K$ has a power basis. Our result gives rise to a large class of infinite towers of monogenic, and in general non-abelian, number fields. Consequently, for any (odd) prime $\ell$ and positive integer $n$, there are infinitely many monogenic number fields of degree $\ell^n$. For a discussion of previous results regarding monogenic number fields, see Narkiewicz \cite{Narkiewicz04}.

\begin{defn}
We say that an order $\cal O \subset \cal O_K$ is $p$-maximal if $p \nmid [\cal O_K \colon \cal O]$.
\end{defn}

Our result is a consequence of the following theorem by Dedekind, which appears in Cohen \cite{Cohen95}.

\begin{thm}[Dedekind's criterion] Let $\bbq(\theta)$ be a number field, $T \in \bbz[X]$ the monic minimal polynomial of $\theta$ and let $p$ be a prime number. Denote by $\overline{\phantom{t}}$ reduction modulo $p$. Let 
$$\overline T(X) = \prod_{i=1}^l \overline{t_i}(X)^{e_i}$$
be the factorization of $T(X)$ modulo $p$ in $\bbf_p[X]$, and set 
$$g(X) = \prod_{i=1}^l t_i(X)$$
where the $t_i \in \bbz[X]$ are arbitrary monic lifts of $\overline{t_i}$. Then
\begin{enumerate}
\item The $p$-radical $I_p$ of $\bbz[\theta]$ at $p$ is given by
$$I_p = p\bbz[\theta] + g(\theta)\bbz[\theta].$$
In other words, $x= A(\theta)\in I_p$ if and only if $\overline g \mid \overline A$.

\item Let $h(X) \in \bbz[X]$ be a monic lift of $\overline T(X)/\overline g(X)$ and set
$$f(X) = \frac{g(X)h(X) - T(X)}{p} \in \bbz[X].$$
Then $\bbz[\theta]$ is $p$-maximal if and only if $\gcd(\overline f, \overline g, \overline h) = 1$ in $\bbf_p[X]$.

\item More generally, if $U$ is a monic lift of $\overline T /(\overline f, \overline g, \overline h)$ to $\bbz[X]$, we have
$$\cal O' := \bbz[\theta] + \frac{1}{p}U(\theta)\bbz[\theta]$$
and if $m = \deg(\overline f, \overline g, \overline h)$, then $[\cal O'\colon \bbz[\theta]] = p^m$, hence $\disc(\cal O') = \disc(T)/p^{2m}$.
\end{enumerate}
\end{thm}

\begin{remark}
The ring $\cal O'$ in part (3) of Dedekind's criterion is subring of the ring of integers $\cal O_{\bbq(\theta)}$. However, $\cal O'$ is not necessarily $p$-maximal, i.e. $[{\cal O}_K \colon \cal O']$ may be divisible by $p$.
\end{remark}

We now work towards our main result regarding towers of monogenic extensions. The first result specifies conditions on $t$ for which $K$ is monogenic.

\begin{thm} \label{th:dc1} 
If $\Phi(x)=T_\ell^n(x)-t$ is irreducible, then $D(\Phi) = \Delta(K)$ if and only if 
\begin{enumerate}
\item $\Phi(t) \not\equiv 0 \mmod {\ell^2}$ and
\item both $t-2$ and $t+2$ are square-free.
\end{enumerate}
\end{thm}

\begin{proof} 
Let $\theta$ be a root of $\Phi$. The discriminants $\Delta(K)$ and $D(\Phi)$ are equal if and only if $\bbz[\theta]$ is $p$-maximal for every prime $p$. We do not need to check every prime; the only primes for which $\bbz[\theta]$ may not maximal are the primes that divide $D(\Phi)$ with multiplicity at least 2.

By Proposition \ref{prop:disc}, is sufficient to check $\ell$ and the primes that divide $t^2-4$.

We begin by using Dedekind's criterion to identify the condition on $t$ for which $\bbz[\theta]$ is $\ell$-maximal. By Lemma \ref{lem:chebymodl},
$$\Phi(x) \equiv x^{\ell^n} - t \equiv (x - t)^{\ell^n} \pmod \ell,$$
and we write
$$g(x) = x - t, \quad h(x) = (x - t)^{\ell^n-1}, \quad f(x) = \frac{(x - t)^{\ell^n} - \Phi(x)}{\ell}.$$
The ring $\bbz[\theta]$ is $\ell$-maximal if and only if $\gcd(\overline f, \overline g, \overline h) = 1$, which holds if and only if $t$ is not a root of $\overline f$ modulo $\ell$. Evaluating $f(t)$, we see that
$$f(t) = \frac{\Phi(t)}{\ell} \not\equiv 0 \pmod \ell, \text{\quad and equivalently \quad} \Phi(t) \not \equiv 0 \pmod {\ell^2}.$$

Now, let $p$ be a prime dividing $t^2-4$, i.e. $p \mid (t-2)(t+2)$. In this case, $t \equiv \pm2 \mmod p$, and we write $\overline t \in \{2,-2\}$ for the reduction of $t$ modulo $p$. It follows from Lemma \ref{lem:2.1} that $\Phi(x) \equiv (x-\overline t)\tau(x)^2 \mmod p$, where $T_\ell^n(x)-\overline t = (x-\overline t)\tau(x)^2$. We seek to apply Dedekind's criterion, so we write
$$g(x) = (x-\overline t) \tau(x), \quad h(x) = \tau(x), \quad f(x) = \frac{(x-\overline t)\tau(x)^2 - \Phi(x)}{p},$$
and proceed to show that $\gcd(\overline f, \overline g, \overline h) =1$.

Let $\alpha$ be a root of $\tau$ modulo $p$. Then $\gcd(\overline f, \overline g, \overline h) = 1$ if and only if $\alpha$ is not a root of $f$ modulo $p$. Evaluating $f$ at $\alpha$, we see that
$$f(\alpha) = -\frac{\Phi(\alpha)}{p}\equiv 0 \pmod p \text{ \quad if and only if \quad } \Phi(\alpha) \equiv 0 \pmod{p^2}.$$
Note that
\begin{align*}
\Phi(\alpha) = T_\ell^n(\alpha)-t = T_\ell^n(\alpha)-\overline t + \overline t - t  = (x-\overline t)\tau(\alpha)^2 + \overline t - t \equiv t-\overline t \pmod{p^2}.
\end{align*}
Thus $\bbz[\theta]$ is $p$-maximal if and only if $t - \overline t \not\equiv 0 \mmod{p^2}$, concluding the proof.
\end{proof}

The remainder of this section is dedicated to expanding this result.

\begin{prop} \label{prop:dc2} For any integers $a$ and $b$, 
$$T_\ell(a) \equiv T_\ell(b) \pmod{\ell^2} \mbox{ if and only if } a \equiv b \pmod \ell.$$
\end{prop}

\begin{proof} Suppose that $T_\ell(a) \equiv T_\ell(b) \mmod{\ell^2}$. By Lemma \ref{lem:chebymodl}, $T_\ell(x) = x^\ell + \ell \cdot g(x)$, where $g(x)$ is a polynomial of degree $\ell-2$. Hence
\begin{align*}
T_\ell(a) \equiv T_\ell(b) \pmod{\ell^2} &\Rightarrow a^\ell + \ell \,g(a) \equiv b^\ell + \ell \,g(b) \pmod{\ell^2}\\
&\Rightarrow a^\ell \equiv b^\ell \pmod \ell\\
&\Rightarrow a \equiv b \pmod \ell.
\end{align*}
For the converse statement, let $a \in \bbz$ and write $a = q\ell + r$ such that $0 \le r < \ell$. It suffices to show that $T_\ell(a) \equiv T_\ell(r) \mmod{\ell^2}$.
\begin{align*}
T_\ell(a) = T_\ell(q\ell + r) &= \sum_{k=0}^{\floor{\ell/2}}(-1)^k \frac{(\ell-k-1)!}{k!(\ell-2k)!} \ell(q\ell+r)^{\ell-2k}\\
&= \sum_{k=0}^{\floor{\ell/2}}(-1)^k \frac{(\ell-k-1)!}{k!(\ell-2k)!}\sum_{i=0}^{\ell-2k}{\ell-2k \choose i}q^i\ell^{i+1}r^{\ell-2k-i}\\
&\equiv \sum_{k=0}^{\floor{ \ell/2}}(-1)^k \frac{(\ell-k-1)!}{k!(\ell-2k)!} \ell r^{\ell-2k} \\
&\equiv T_\ell(r) \pmod{\ell^2}.
\end{align*}
\end{proof}

We are now ready to prove the main theorem of this section. 

\begin{thm}\label{th:monogenic}
If $\Phi(x) = T_\ell^n(x)-t$ is irreducible, then $D(\Phi) = \Delta(K)$ if and only if 
\begin{enumerate}
\item $T_\ell(t) - t \not\equiv 0 \mmod {\ell^2}$ and
\item both $t-2$ and $t+2$ are square-free.
\end{enumerate}
\end{thm}

\begin{proof} Note that for $n \ge 1$, $T^{n-1}_\ell(t) \equiv t^{\ell^{n-1}} \equiv t \mmod \ell$. By Proposition \ref{prop:dc2}, 
$$T^n_\ell(t) = T_\ell(T^{n-1}_\ell(t)) \equiv T_\ell(t) \pmod{\ell^2}.$$
Thus 
$$T^n_\ell(t) \equiv t \pmod{\ell^2} \mbox{\quad if and only if \quad } T_\ell(t) \equiv t \pmod{\ell^2}.$$ The result is now an immediate consequence of Theorem \ref{th:dc1}.
\end{proof}

\begin{remark}
What we have shows is that the conditions for which $D(\Phi) = \Delta(K)$ only depend on $\ell$ and $t$. By picking a compatible sequence of preimages of $t$, $\{t=\theta_0, \theta_1, \ldots\}$, such that $T_\ell(\theta_n) = \theta_{n-1}$, we designate a tower of number fields
$$\bbq \subset K_1 \subset K_2 \subset \cdots$$
where $K_n = \bbq(\theta_n)$. By our result, $K_1$ is monogenic if and only if $K_n$ is monogenic, and in particular we have identified a two parameter family of towers of monogenic number fields.
\end{remark}

We conclude this section by identifying the values $t$ for which $\Phi(t)\equiv 0 \mmod{\ell^2}$.

\begin{thm} \label{th:zeromodellsquared}
$\Phi(t) \equiv 0 \mmod {\ell^2}$ if and only if $t \equiv T_\ell(a)$ for some $a \in \bbz/\ell^2\bbz$.
\end{thm}

\begin{proof} 
By the same argument in the proof of Theorem \ref{th:monogenic}, it is sufficient to show that $T_\ell(t) \equiv t \mmod{\ell^2}$ if and only if $T_\ell(a) \equiv t \mmod{\ell^2}$ for some $a \in \bbz/\ell^2\bbz$.

Suppose there exists $a \in \bbz/\ell^2\bbz$ such that $T_\ell(a) \equiv t \mmod{\ell^2}$. By Lemma \ref{lem:chebymodl}, $T_\ell(a) \equiv a \mmod \ell$, and so $a \equiv t \mmod \ell$. Thus, by Proposition \ref{prop:dc2}
$$T_\ell(t) \equiv T_\ell(a) \equiv t \pmod{\ell^2}.$$
The converse statement is immediate by setting $a=t$.
\end{proof}

\begin{remark}
In fact, Proposition \ref{prop:dc2} implies that $T_\ell(x)$ is an $\ell$-to-one map from $\bbz/\ell^2\bbz$ to $\bbz/\ell^2\bbz$ defined by $a+b\ell \mapsto T_\ell(a)$. Thus for every prime $\ell$, there are exactly $\ell$ ``bad" values for $t$ modulo $\ell^2$ for which $K$ is not monogenic. These bad values are exactly the set of values 
$$\{T_\ell(0),T_\ell(1),\ldots, T_\ell(\ell-1)\} \subset \bbz/\ell^2\bbz.$$
We note that $2 = T_\ell(2)$ and $-2 \equiv T_\ell(\ell-2) \mmod{\ell^2}$ are necessarily contained in this set.
\end{remark}

\begin{cor}
For an arbitrary choice of $t$, the probability that $\ell \nmid \ind(\Phi)$ is $1-\ell\inv$.
\end{cor}

\begin{prop}
For an arbitrary choice of $t$, the probability that $D(\Phi) = \Delta(K)$ is
\begin{align*}
\frac{\ell^2-\ell+2}{\ell^2}\cdot\frac{6}{\pi^2}\prod_{p \text{\rm\tiny{} odd prime }}\left(1-\frac{1}{p^2-1}\right).
\end{align*}
\end{prop}

\begin{proof}
It is well known that the probability that $t+2$ (or any other integer) is square-free is $6/\pi^2$. Given that $t+2$ is square-free, $t-2$ is square-free only if $t \not\equiv 2 \mmod{p^2}$ for any prime $p$. Hence, the probability that $t-2$ is square-free given that $t+2$ is square-free is
\begin{align*}
\prod_{p \text{\tiny{} \rm odd prime}}\left(1-\frac{1}{p^2-1}\right).
\end{align*}
Lastly, $t$ cannot be any of the other $\ell-2$ equivalence classes modulo $\ell^2$ identified above, and the result follows.
\end{proof}

\section{Montes algorithm: Theorem of the index} \label{sec:gmn method}
The remainder of the paper is dedicated to studying the cases where $D(\Phi) \neq \Delta(K)$. This will happen whenever the conditions in Theorem \ref{th:monogenic} are relaxed, namely, if $t$ is chosen so that  $T_\ell(t)-t \equiv 0 \mmod {\ell^2}$, and/or at least one of $t-2$ and $t+2$ is not square-free. In this section we introduce the Montes algorithm for computing $\ind(f)$ that was recently developed by Gu\`ardia, Montez, and Nart \cite{GMN08,GMN09,GMN12}, though we primarily follow the presentation found in the paper by el Fadil, Montez, Nart \cite{EMN09}. Their method employs a more refined variation of the Newton polygon, called the $\phi$-Newton polygon, which captures arithmetic data attached to the irreducible factor $\phi$ of $\overline\Phi$.

\begin{notation} We fix the following notation. Let $p$ be a prime number and let $\phi(x) \in \bbz[x]$ be a monic polynomial whose reduction modulo $p$ is irreducible. We denote by $\bbf_\phi$ the finite field $\bbz[x]/(p,\phi)$, and by 
$$\overline{\phantom{M}} \colon \bbz[x] \to \bbf_p[x], \quad \red\colon \bbz[x] \to \bbf_\phi$$
the respective homomorphisms of reduction modulo $p$ and modulo $(p,\phi(x))$. We extend the usual $p$-adic valuation to polynomials by
$$\nu_p(c_0 + \cdots + c_rx^r) := \min_{0\le i\le r}\{\nu_p(c_i)\}.$$
Any $f(x) \in \bbz[x]$ admits a unique $\phi$-adic development:
$$f(x) = a_0(x) + a_1(x)\phi(x) + \cdots + a_r(x) \phi(x)^r,$$
with $a_i(x) \in \bbz[x]$ and $\deg(a_i) < \deg(\phi)$. To each coefficient $a_i(x)$ we attach the $p$-adic value 
$$u_i = \nu_p(a_i(x))\in \bbz\cup \{\infty\}$$
and the point of the plane $(i,u_i)$, if $u_i < \infty$. 
\end{notation}

\begin{defn}
The \emph{$\phi$-Newton polygon} of $f(x)$ is the lower convex envelope of the set of points $(i,u_i), u_i < \infty$, in the Euclidian plane. We denote this open polygon by $N_\phi(f)$.
\end{defn}

\noindent 
The $\phi$-Newton polygon is the union of different adjacent sides $S_1, \ldots, S_g$ with increasing slopes $\lambda_1 < \cdots < \lambda_g$. We shall write $N_\phi(f) = S_1 + \cdots +S_g$. The end points of the sides are called the vertices of the polygon.

\begin{defn}
The polygon determined by the sides of negative slope of $N_\phi(f)$ is called the \emph{principal $\phi$-polygon} of $f(x)$ and will be denoted by $N_\phi^-(f)$. The length, of $N_\phi^-(f)$, denoted $\ell(N_\phi^-(f))$, is always equal to the highest exponent $a$ such that $\overline{\phi(x)}^a$ divides $\overline{f(x)}$ in $\bbf_p[x]$.
\end{defn}

\begin{notation} From now on, any reference to the $\phi$-Newton polygon of $f(x)$ will be taken to mean the principal $\phi$-polygon, and for simplicity, we will write $N_\phi(f) := N_\phi^-(f)$. 
\end{notation}

We attach to any abscissa $0 \le i \le \ell(N_\phi)$ the following residual coefficient $c_i \in \bbf_p[x]/(\phi)$.
$$c_i = \begin{cases} 0 & \text{ if $(i,u_i)$ lies strictly above $N_\phi$ or $u_i = \infty$},\\ \red(a_i(x)/p^{u_i}) & \text{ if $(i,u_i)$ lies on $N_\phi$.}\end{cases}$$
Note that $c_i$ is always nonzero in the latter case, because $\deg(a_i(x)) < \deg(\phi)$.

Let $S$ be one of the sides of $N_\phi$, with slope $\lambda = -h/e$, where $e$ and $h$ are relatively prime, positive integers. The length of $S$ is the length, $\ell(S)$, of the projection of $S$ to the horizontal axis, the degree of $S$ is $d(S) := \ell(S)/e$, the ramification index of $S$ is $e(S) := e$.

\begin{defn}
Let $s$ be the initial abscissa of $S$, and let $d := d(S)$. We define the \emph{residual polynomial} attached to $S$ (or to $\lambda$) to be the polynomial
$$R_\lambda(f)(y) := c_s+ c_{s+e}y + \cdots + c_{s+(d-1)e}y^{d-1} + c_{s+de}y^d \in \bbf_\phi[y].$$
\end{defn}

\begin{defn} \label{def:4.4}
Let $\phi(x) \in \bbz[x]$ be a monic polynomial, irreducible modulo $p$. We say that $f(x)$ is \emph{$\phi$-regular} if for every side $N_\phi(f)$, the residual polynomial attached to the side is separable. 

Choose monic polynomials $\phi_1(x), \ldots, \phi_r(x) \in \bbz[x]$ whose reduction modulo $p$ are the different irreducible factors of $\overline{f(x)} \in \bbf_p[x]$. We say that $f(x)$ is \emph{$p$-regular} with respect to this choice if $f(x)$ is $\phi_i$-regular for every $1 \le i \le r$.
\end{defn}

\begin{defn}
The \emph{$\phi$-index} of $f(x)$ is $\deg \phi$ times the number of points with integral coordinates that lie below or on the polygon $N_\phi(f)$, strictly above the horizontal axis, and strictly to the right of the vertical axis. We denote this number by $\ind_\phi(f)$.
\end{defn}

\begin{notation} Let $f(x) \in \bbz[x]$ be a monic irreducible polynomial and let $\theta$ be a root of $f(x)$. We denote by
$$\ind_p(f) := \nu_p([\cal O_{\bbq(\theta)} \colon \bbz[\theta]]),$$
the $p$-adic value of the index of the polynomial $f(x)$.
\end{notation}

\begin{thm}[\cite{GMN12}, section 4.4] \label{th:index}
Theorem of the index:$$\ind_p(f) \ge \ind_{\phi_1}(f) + \cdots + \ind_{\phi_r}(f),$$
and equality holds if $f(x)$ is $p$-regular.
\end{thm}

\section{Index calculations: On the multiplicity of $\ell$} \label{sec:ell-index}
Recall that $\Phi(x):=T_\ell^n(x)-t$, and let $\theta$ be a root of $\Phi$. In the proof of Theorem \ref{th:dc1} we showed that $\bbz[\theta]$ is $\ell$-maximal if and only if $\Phi(t) \not\equiv 0 \mmod{\ell^2}$. Here, we relax this condition and study the effect on $\ind_\ell(\Phi)$. Specifically, in this section, we fix $t$ so that $\Phi(t)\equiv 0 \mmod{\ell^2}$ \emph{with the exception} that $t \not \equiv \pm 2 \mmod{\ell^2}$.

Following the prescription outlined in the previous section, we must start by factoring $\Phi$ modulo $\ell$. Recalling Lemma \ref{lem:chebymodl}:
$$\Phi(x) \equiv (x-t)^{\ell^n}\pmod\ell.$$
If $t \equiv 0 \mmod \ell$, then $\phi(x) = x$, and $N_\phi(\Phi)$ is just the usual Newton polygon of $\Phi$. Otherwise, if $t \not\equiv 0 \mmod \ell$, then $\phi(x) = x-t$, and we must compute the $\phi$-development of $\Phi$. Note, however, that $N_\phi(\Phi)$ is the Newton polygon of the shifted Chebyshev polynomial $\Phi(\phi(x)+t)$ as a polynomial in $\phi(x)$. The following lemma will assist our calculations.

\begin{defn}
For any prime $p$ and any integer $a$, the \emph{$p$-adic expansion} of $a$ is
$$a = a_0p^0 + a_1 p^1 + a_2 p^2 + \cdots + a_s p^s$$
with $0 \le a_i < p$. We define the function
$$\sigma_p(a) = \sum_{i=0}^\infty a_i.$$
\end{defn}

\begin{lem}\label{lem:kummer}
Let $p$ be a prime, and let $\sigma_p$ be the function defined above.
\begin{enumerate}
\item Let $a$ and $b$ be integers written in base $p$. The number of ``carries" performed when summing $a+b$ in base $p$ is
$$\# \text{carries} = \frac{\sigma_p(a) + \sigma_p(b) - \sigma_p(a+b)}{p-1}.$$

\item $\ds\nu_p(a) = \frac{1 + \sigma_p(a-1) - \sigma_p(a)}{p-1}.$

\item $\ds\nu_p(a!) = \frac{n-\sigma_p(a)}{p-1}.$

\item (Kummer \cite{Kummer52}) $\ds\nu_p{a+b \choose b} = \#\text{carries in $a+b$ summed in base $p$.}$
\end{enumerate}
\end{lem}

Though these are well-known, for the convenience of the reader, we provide proofs, as they are short.

\begin{proof} 
\
\begin{enumerate}
\item Write $a$ and $b$ in their base $p$ expansions: $a = \sum a_i p^i$ and $b =\sum b_i p^i$. If ever $c_i := a_i + b_i \ge p$, then perform a ``carry": subtract $p$ from $c_i$ and add 1 to $c_{i+1}$, repeating until all $c_i$ are less than $p$. These $c_i$ are the coefficients for the base $p$ expansion of $a+b$: $a+b = \sum c_ip^i$. Each carry reduces the sum $\sigma_p(a) + \sigma_p(b)$ by $p-1$, and the result follows. 

\item This follows immediately from part (1). If $k$ is the smallest integer for which $a-1 \equiv -1 \mmod{p^k}$, then the sum $(a-1)+1$ requires $k$ carries in base $p$.

\item By part (2), we have the telescoping sum
\begin{align*}
\nu_p(a!) = \sum_{i=1}^a \nu_p(i) = \sum_{i=1}^a \frac{1+\sigma_p(i-1) - \sigma_p(i)}{p-1} = \frac{a - \sigma_p(a)}{p-1}.
\end{align*}

\item By part (3)
\begin{align*}
\nu_p{a+b \choose b} &= \nu_p\left(\frac{(a+b)!}{a!b!}\right) = \nu_p((a+b)!) - \nu_p(a!) - \nu(b!) \\
&= \frac{a+b - \sigma_p(a+b)}{p-1} + \frac{a-\sigma_p(a)}{p-1} - \frac{b-\sigma_p(b)}{p-1} \\
&= \frac{\sigma_p(a) + \sigma_p(b) - \sigma_p(a+b)}{p-1}.
\end{align*}
The result follows from part (1).
\end{enumerate}
\end{proof}

We consider the case where $t \equiv 0 \mmod\ell$ and proceed by computing the Newton polygon of $T_\ell^n(x)$. We also write $\ds T_\ell^n(x) = \sum_{k=0}^{\floor{ \ell^n/2}}c_ix^{\ell^n-2k}$ where
$$c_i:= \frac{\ell^n}{\ds\ell^n-(\ell^n-i)/2}{\ell^n-(\ell^n-i)/2 \choose \ds(\ell^n-i)/2} = \frac{2\ell^n}{\ell^n+i}{(\ell^n+i)/2\choose (\ell^n-i)/2}.$$ 

\begin{prop} \label{prop:newtonpolyofchebi}
For any integer $0 < i \le \ell^m \le \ell^n$, $\nu_\ell(c_i) \ge n-m$ with equality only if $i = \ell^m$. Furthermore, $N_\phi(T_\ell^n) = \sum_{m=1}^n S_m$ where $S_m$ is the edge with endpoints $(\ell^{m-1},n-m+1)$ and $(\ell^m,n-m)$.
\end{prop}

\begin{proof}
When $i = \ell^m$,
\begin{align*}
\nu_\ell(c_{\ell^m}) &= n + \nu_\ell{(\ell^n+\ell^m)/2 \choose (\ell^n-\ell^m)/2} - \nu_\ell(\ell^n+\ell^m).\nonumber
\end{align*}
Note that 
\begin{align*}
{(\ell^n+\ell^m)/2 \choose (\ell^n-\ell^m)/2} = {(\ell^n+\ell^m)/2 \choose (\ell^n+\ell^m)/2 - (\ell^n-\ell^m)/2} = {(\ell^n+\ell^m)/2 \choose \ell^m}.
\end{align*}
The $\ell$-valuation of this number can be determined using Lemma \ref{lem:kummer} by considering a sum in base $\ell$. Writing
$$\frac{\ell^n+\ell^m}{2} - \ell^m = \frac{\ell-1}{2} \cdot\ell^m + \frac{\ell-1}{2}\cdot \ell^{m+1} + \cdots +\frac{\ell-1}{2}\cdot \ell^n,$$
it is clear that $(\frac{\ell^n-\ell^m}{2}-\ell^m) + \ell^m$ requires no carries when summed in base $\ell$. Thus 
$$\nu_\ell{(\ell^n+\ell^m)/2 \choose (\ell^n-\ell^m)/2}=0.$$

Furthermore, 
\begin{align*}
\nu_\ell(\ell^n+\ell^m) = \nu_\ell(\ell^m) + \nu_\ell(\ell^{n-m}+1)= m,
\end{align*}
proving that $\nu_\ell(c_{\ell^m}) = n-m.$

Suppose that $0 < i < \ell^m$. Then $\nu_\ell(\ell^n+i) = \nu_\ell(i) < m$, and 
$$\nu_\ell(c_i) = n +\nu_\ell{(\ell^n+i)/2 \choose (\ell^n-i)/2} - \nu_\ell(\ell^n+i) > n - m.$$

From the previous parts we conclude that the polygon $S_1+ \cdots+S_n$ is a lower boundary for the points $(i,\nu_\ell(c_i))$ with vertices at $(m,\nu_\ell(c_{\ell^m}))$. It is easily verified that this polygon is convex. Let $\lambda_m$ denote the slope of $S_m$, which is given by
$$\lambda_m = \frac{-1}{\ell^{m-1}(\ell-1)}.$$
Clearly, $\lambda_1 < \lambda_2 < \cdots < \lambda_n$.
\end{proof}

Knowing the Newton polygon for $T_\ell^n$, we can determine the Newton polygon for $\Phi$.

\begin{prop}
Suppose $t\equiv 0 \mmod \ell$, and let $v=\nu_\ell(t)$. Let $S_m, m=1, \ldots, n$ be the edges defined in the previous theorem. Define $S'$ to be the edge with endpoints $(0,v)$ and $(\ell^{n-v+1},v-1)$. Then 
$$N_\phi(\Phi) = S' + S_{n-v+2}+S_{n-v+3} + \cdots+S_n.$$
\end{prop}

\begin{proof}
Let $\lambda_m$ be the slope of $S_m$ and $\lambda'$ be the slope of $S'$. It suffices to show that $\lambda_{n-v+1}<\lambda'<\lambda_{n-v+2}$. This is easily verified:
\begin{align*}
\lambda_{n-v} = \frac{-1}{\ell^{n-v}(\ell-1)} \quad < \quad \lambda' = \frac{-1}{\ell^{n-v+1}} \quad < \quad \lambda_{n-v+2} = \frac{-1}{\ell^{n-v+1}(\ell-1)}.
\end{align*}
\end{proof}

\begin{remark}
Although we write $v = \nu_\ell(t)$, due to the nature of the coefficients of $\Phi(x)$ it would also have been correct to define $v = \nu_\ell(\Phi(t))$. We adopt this new definition of $v$ for the remainder of this section.
\end{remark}

We move on to the case where $\phi(x) = x-t$ and $t \not\equiv 0 \mmod \ell$ where we must determine the Newton polygon of $\Phi(x) = \Phi(\phi(x)+t)$ as a polynomial in $\phi(x)$. We proceed by determining the $\phi$-development of $\Phi(x) = T_\ell^n(x)-t$.
\begin{align*}
T_\ell^n(\phi+t) - t &= -t + \sum_{k=0}^{\floor{ \ell^n/2}}(-1)^k{\ell^n-k \choose k}\frac{\ell^n}{\ell^n-k}(\phi+t)^{\ell^n-2k} \\
&= -t + \sum_{k=0}^{\floor{ \ell^n/2}}(-1)^k{\ell^n-k \choose k}\frac{\ell^n}{\ell^n-k}\sum_{i=0}^{\ell^n-2k} {\ell^n-2k\choose i} t^{\ell^n-2k-i}\phi^i\\
&= -t+\sum_{i=0}^{\ell^n} \sum_{k=0}^{\floor{ (\ell^n-i)/2}} (-1)^k{\ell^n-k\choose k}{\ell^n-2k\choose i}\frac{\ell^n}{\ell^n-k}t^{\ell^n-2k-i}\phi^i \\
&= -t + \sum_{k=0}^{\floor{ \ell^n/2}}(-1)^k{\ell^n-k \choose k} \frac{\ell^n}{\ell^n-k}t^{\ell^n-2k} \\
&\hspace{2cm}+\sum_{i=1}^{\ell^n} \sum_{k=0}^{\floor{ (\ell^n-i)/2}} (-1)^k{\ell^n-k\choose k}{\ell^n-2k\choose i}\frac{\ell^n}{\ell^n-k}t^{\ell^n-2k-i}\phi^i \\
&= T_\ell^n(t) - t + \sum_{i=1}^{\ell^n} \sum_{k=0}^{\floor{ (\ell^n-i)/2}} (-1)^k{\ell^n-k\choose k}{\ell^n-2k\choose i}\frac{\ell^n}{\ell^n-k}t^{\ell^n-2k-i}\phi^i
\end{align*}
For ease, we will write
$$b_i := \ell^n\sum_{k=0}^{\floor{\frac{\ell^n-i}{2}}}(-1)^k{\ell^n-k\choose k}{\ell^n-2k \choose i} \frac{t^{\ell^n-2k-i}}{\ell^n-k}$$
denote the coefficient of $\phi^i$ for $1 \le i \le \ell^n$.

\begin{lem}\label{lem:binomialrelation}
For positive integers $a, b$, and $c$ satisfying $0 \le b \le \frac{a-c}{2}$, the binomial coefficients satisfy the following relationship:
\begin{align*}
{a-b\choose b}{a-2b\choose c} = {a-b-c \choose b}{a-b \choose c}.
\end{align*}
\end{lem}

\begin{proof}
\begin{align*}
{a-b\choose b}{a-2b\choose c} &= \frac{(a-b)!}{b!(a-2b!)} \cdot \frac{(a-2b)!}{c!(a-2b-c)!} = \frac{(a-b)!}{c!(a-b-c)!}\cdot\frac{(a-b-c)!}{b!(a-2b-c)!} \\ &={a-b-c\choose b}{a-b\choose c}.
\end{align*}
\end{proof}

This lemma allow us to rewrite $b_i$ in a way that is better suited for our analysis:
\begin{align*}
b_i &= \sum_{k=0}^{\floor{\frac{\ell^n-i}{2}}} (-1)^k \frac{\ell^n}{\ell^n-k}{\ell^n-k\choose k}{\ell^n-2k \choose i}t^{\ell^n-2k-i} \\
&= \frac{\ell^n}{i} \sum_{k=0}^{\floor{\frac{\ell^n-i}{2}}}(-1)^k{\ell^n-i-k \choose k}{\ell^n-k -1\choose i-1}t^{\ell^n-2k-i}.
\end{align*}

There are a couple more results that we will use to compute the valuations of the coefficients. The first is a result by Lucas \cite{Lucas78}.

\begin{thm}[Lucas] \label{th:lucas}
Let $p$ be a prime, and let $0 \le m \le n$ with $n = \sum_{j=0}^l n_j p^j$ and $m = \sum_{j=0}^l m_jp^j$. Then
$${n \choose m} \equiv \prod_{j=0}^l{n_j \choose m_j}\pmod p.$$
\end{thm}

We also use the following fact about Chebyshev polynomials (of the second kind) $U_d(x)$. Recall that $U_d(x)$ is given by the expression
$$U_d(x) = \frac{\left(x+\sqrt{x^2-4}\right)^{d+1} - \left(x-\sqrt{x^2-4}\right)^{d+1}}{2^{d+1}\sqrt{x^2-4}} \quad \text{ if } \quad x \neq \pm 2.$$

\begin{lem} \label{lem:modl}
Let $\ell$ be an odd prime. If $x \not\equiv \pm2 \mmod \ell$, then $U_{\ell-1}(x) \equiv \pm 1 \mmod \ell$.
\end{lem}

\begin{proof}
A local calculation: Let $\alpha = \frac{x+\sqrt{x^2-4}}{2}$ and $\beta = \frac{x-\sqrt{x^2-4}}{2}$, and consider $\alpha$ and $\beta$ as elements of $\bbf_{\ell^2}$. Here, the Frobenius map fixes $\bbf_\ell$, and acts by conjugation on its complement. Hence, if $\sqrt{x^2-4} \in \bbf_\ell$, then $\alpha^\ell = \alpha$ and $\beta^\ell = \beta$, so
\begin{align*}
U_{\ell-1}(x) = \frac{\alpha - \beta}{\sqrt{x^2-4}} = 1 \in \bbf_\ell.
\end{align*}
Otherwise, if $\sqrt{x^2-4} \not \in \bbf_\ell$, then $\alpha^\ell = \beta$ and $\beta^\ell = \alpha$, so
\begin{align*}
U_{\ell-1}(x) = \frac{\beta-\alpha}{\sqrt{x^2-4}} = -1 \in \bbf_\ell.
\end{align*}
\end{proof}

We proceed to compute the valuations of the coefficients $b_i$.
\begin{thm}
Suppose that $t \not \equiv \pm2 \mmod \ell$  and $\ell^m \le i < \ell^{m+1} \le \ell^n$. Then $\nu_\ell(b_i) \ge n-m$ with equality if $i = \ell^m$. 
\end{thm}

\begin{proof}
Assume first that $i = \ell^m + \varepsilon$ for some integer $0 <\varepsilon < (\ell-1)\ell^m$. We show that $\nu_\ell(b_i) \ge n-m$. Note that
\begin{align*}
{\ell^n-k-1 \choose \ell^m+\varepsilon-1} &= \frac{(\ell^n-k-1)!}{(\ell^m+\varepsilon)!(\ell^n-\ell^m-k-\varepsilon)!}\\[1mm]
&=\frac{(\ell^n-k-1)!}{\ell^m(\ell^m-1)!(\ell^n-\ell^m-k)!}\frac{\ell^m!\varepsilon!}{(\ell^m+\varepsilon)!}\frac{(\ell^n-\ell^m-k)!}{\varepsilon!(\ell^n-\ell^m-k-\varepsilon)!}\\[1mm]
&= \frac{{\ell^n-k-1 \choose \ell^m-1}{\ell^n-\ell^m-k \choose \varepsilon}}{{\ell^m+\varepsilon \choose \ell^m}}.
\end{align*}
Hence,
\begin{align*} 
b_i &= \frac{\ell^n}{\ell^m+\varepsilon}\sum_{k=0}^{\floor{\frac{\ell^n-i}{2}}}(-1)^k{\ell^n-i-k\choose k}{\ell^n-k-1 \choose \ell^m+\varepsilon-1}t^{\ell^n-2k-i}\nonumber\\
&= \frac{\ell^{n-m}}{{\ell^m+\varepsilon \choose \ell^m}}\sum_{k=0}^{\floor{\frac{\ell^n-i}{2}}}(-1)^k{\ell^n-i-k\choose k}{\ell^n-k-1 \choose \ell^m-1}{\ell^n-\ell^m-k \choose \varepsilon}t^{\ell^n-2k-i}.
\end{align*}
The valuation of the sum is non-negative since it is an integer. Furthermore, $\nu_\ell{\ell^m+\varepsilon \choose \ell^m} = 0$ by Lemma \ref{lem:kummer} since $\ell^m+\varepsilon$ requires no carries in base $\ell$. Thus $\nu_\ell(b_i) \ge n-m.$

Assume now that $i = \ell^m$, i.e.
\begin{align} \label{eq:5.1}
b_{\ell^m} = \ell^{n-m}\sum_{k=0}^{\frac{\ell^n-\ell^m}{2}} (-1)^k {\ell^n-\ell^m-k\choose k}{\ell^n-k-1 \choose \ell^m-1}t^{\ell^n-\ell^m-2k}.
\end{align}
To show that $\nu_\ell(b_{\ell^m}) = n-m$, we show that $\ell^{m-n}b_{\ell^m}$ is relatively prime to $\ell$. It suffices to sum over the terms that are relatively prime to $\ell$ and show that the sum of these terms is not divisible by $\ell$. We write the following numbers in their base-$\ell$ expansions.
\begin{align*}
k &= \sum_{j=0}^{n-1} k_j\ell^j; && \ell^m-1 = \sum_{j=0}^{m-1}(\ell-1)\ell^j; && \ell^n-k-1 = \sum_{j=0}^{n-1} (\ell-k_j-1)\ell^j.
\end{align*}
By Theorem \ref{th:lucas}, the second binomial coefficient in equation \eqref{eq:5.1} satisfies
\begin{align*}
{\ell^n-k-1\choose \ell^m-1} &\equiv {\ell-k_0-1\choose \ell-1}\cdots{\ell-k_{m-1}-1\choose \ell-1}{\ell-k_m-1\choose0} \cdots {\ell-k_{n-1}-1\choose 0}\pmod \ell\\
&\equiv \begin{cases} 1 \pmod \ell& \text{ if }k_0 = \cdots = k_{m-1} = 0 \\ 0 \pmod \ell& \text{ otherwise.} \end{cases}
\end{align*}
That is, ${\ell^n-k-1\choose \ell^m-1}$ is relatively prime to $\ell$ if and only if $\ell^m \mid k$. We continue with the additional assumption that $\ell^m$ divides $k$. Now, the base-$\ell$ expansion of $\ell^n-\ell^m -k$ is
$$\ell^n-\ell^m - k = \sum_{j=m}^{n-1}(\ell-k_j-1)\ell^j.$$
Applying Theorem \ref{th:lucas} to the first binomial coefficient in equation \eqref{eq:5.1}, we see that
\begin{align*}
{\ell^n-\ell^m-k\choose k} \equiv {\ell-k_m-1\choose k_m} \cdots {\ell-k_{n-1}-1\choose k_{n-1}}\pmod \ell,
\end{align*}
which is nonzero if and only if $0 \le k_j \le (\ell-1)/2$ for each $j=m,m+1, \ldots, n-1$.
We have the following:
\begin{align*}
\ell^{m-n}b_{\ell^m} &= \sum_{k=0}^{\frac{\ell^n-\ell^m}{2}} (-1)^k{\ell^n-\ell^m-k \choose k}{\ell^n-k-1\choose \ell^m-1}t^{\ell^n-\ell^m-2k}\\
&\equiv \sum_{k=0}^{\frac{\ell^n-\ell^m}{2}}(-1)^k{\ell-k_m-1\choose k_m} \cdots {\ell-k_{n-1}-1\choose k_{n-1}}t^{\ell^n-\ell^m-2k}\\
&\equiv \prod_{j=m}^{n-1}\sum_{k_j=0}^{\frac{\ell-1}{2}}(-1)^{k_j}{\ell-k_j-1\choose k_j}t^{\ell-2k_j-1}\\
&\equiv (U_{\ell-1}(t))^{n-m} \equiv \pm 1 \pmod \ell.
\end{align*}
The second to last step takes advantage of the fact that $t^{\ell^n-\ell^m} \equiv t^{\ell-1} \equiv 1 \mmod \ell$, and the final step follows from Lemma \ref{lem:modl}. This concludes the proof.
\end{proof}

\begin{remark}
The assumption we made at the beginning of this section was that $t \not\equiv \pm2 \mmod{\ell^2}$. However, in order to apply Lemma \ref{lem:modl}, we need to restrict $t$ even further so that $t \equiv \pm2 \mmod \ell$. However, since we are also assuming $\Phi(t) \equiv 0 \mmod {\ell^2}$, these conditions are equivalent thanks to Proposition \ref{prop:dc2}. Specifically,
$$t \equiv \pm2 \pmod{\ell^2} \quad \text{ if and only if } \quad \left\{\begin{array}{l} t \equiv \pm2 \pmod \ell, \text{ and }\\ \Phi(t) \equiv 0 \pmod{\ell^2}\end{array}\right\}.$$
\end{remark}

\begin{remark}
We note that in this case, an alternative method for obtaining the $\phi$-development is given by the Taylor expansion formula:
$$\Phi(x) = \Phi(t) + \Phi'(t)\phi(x) + \frac{1}{2}\Phi''(t)\phi(x)^2 + \cdots + \frac{1}{\ell^n!}\Phi^{(\ell^n)}(t)\phi(x)^{\ell^n}.$$
\end{remark}

The result here is essentially identical to Proposition \ref{prop:newtonpolyofchebi}. In fact, what we have shown is that the Newton polygon of the Chebyshev polynomial $T_\ell^n(x)$ is invariant under shifts modulo the constant term. In particular, the $\phi$-Newton polygon is $\ell$-regular, and thus the Theorem of the Index (Theorem \ref{th:index}) gives an exact value for $\ind_\ell(\Phi)$. We provide a few examples at the end of this section to illustrate the formula.

\begin{cor} \label{cor:5.9}
Let $v=\nu_\ell(\Phi(t))$. Let $S_m$ be the edge connecting $(\ell^{m-1},n-m+1)$ to $(\ell^m, n-m)$ and $S'$ to be the edge connecting $(0,v)$ to $(\ell^{n-v+1},v-1)$, as before. Then 
$$N_\phi(\Phi) = S' + S_{n-v+2}+\cdots+S_n.$$
Moreover, 
$$\ind_\ell(\Phi) = \sum_{i=1}^{\min\{v-1,n\}} \ell^{n-i}.$$
\end{cor}

In particular, we have a precise formula for the discriminant of the number field in the following case. 

\begin{cor} \label{cor:fielddisctpm2sf}
Suppose that $t+2$ and $t-2$ are square-free. Let $v = \nu_\ell(\Phi(t))$. Then
\begin{align*}
\Delta(K) = \ell^{n\ell^n-2\cdot\tiny\ind_\ell(\Phi)}(4-t^2)^{(\ell^n-1)/2}.
\end{align*}
\end{cor}

\begin{example} 
Fix $t_0=3^6\cdot 691=451251$. We note that $t_0+2$ and $t_0-2$ are square-free.

\begin{enumerate}[I.]
\item Consider the polynomial $T_3^3(x)-t_0$. By Theorem \ref{th:zeromodellsquared}, $T_3(t)-t\equiv 0 \mmod 9$ if and only if $t \equiv 0,\pm2 \mmod 9$. Since $\nu_3(t_0) = 6$, it is assured that $T_3^3(t_0) - t_0 \equiv 0 \mmod 9$. By Corollary \ref{cor:5.9}, 
$$\ind_\ell(T_3^3(x)-t_0) = \sum_{i=1}^3 3^{3-i} = 13,$$
and by Corollary \ref{cor:fielddisctpm2sf},
$$\Delta(K_{3,3,t_0}) = 3^{55}(4-t_0^2)^{13}.$$

\begin{itsapicture}
\begin{figure}[h!]
\begin{tikzpicture} [x=.5 cm,y=.5cm]
\foreach \x / \y in {1/3, 3/2, 9/1}
	\draw[dashed,gray!60] (0,\y) -- (\x,\y) -- (\x,0);
\draw (-1,0) -- (28,0)
	(0,-.5) -- (0,6.5)
	(1,-.1) node[below] {1}-- (1,.1)
	(3,-.1) node[below] {3}-- (3,.1)
	(9,-.1) node[below] {9} -- (9,.1)
	(27,-.1) node[below] {27}-- (27,.1);
\foreach \y in {1,...,6}
	\draw (-.2,\y) node[left] {\y} -- (.2,\y);
\filldraw[blue, line width = 2pt] (27,0) circle (2pt)
	(0,6) circle (2pt);
\node[below left] at (0,0) {0};
\draw[blue,line width =2pt] (27,0) -- (9,1) -- (3,2) -- (1,3) -- (0,6);
\foreach \x in {1,...,9}
	\filldraw[orange,line width = 2pt] (\x,1) circle (2pt);
\foreach \x in {1,2,3}
	\filldraw[orange,line width = 2pt] (\x,2) circle (2pt);
\filldraw[orange, line width = 2pt] (1,3) circle (2pt);
\end{tikzpicture}
\caption{The $\phi$-Newton polygon for $T_3^3(x)-t_0$ with $\phi(x) = x$. The polynomial $\phi$ is determined by reducing $T_3^3(x)-t_0$ modulo 3: $T_3^3(x)-t_0 \equiv x^{27} \mmod 3$. The 3-index is computed by counting the integral points lying below the polygon.}
\end{figure}
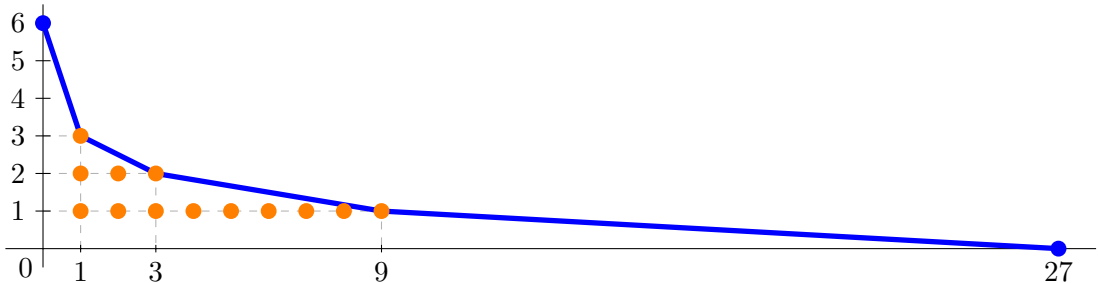
\end{itsapicture}

\item Consider the polynomial $T_5^3(x)-t_0$. By Theorem \ref{th:zeromodellsquared}, $T_5(t)-t \equiv 0 \mmod{25}$ if and only if $t \equiv 0, \pm1, \pm2 \mmod{25}$, and we note that $t_0 \equiv 1 \mmod{25}$. Moreover, $\nu_5(T_5^3(t_0)-t_0) = 4$, so by Corollary \ref{cor:5.9},
$$\ind_5(T_5^3-t_0) = \sum_{i=1}^3 5^{3-i} = 31,$$
and by Corollary \ref{cor:fielddisctpm2sf},
$$\Delta(K_{5,3,t_0}) = 5^{313}(4-t_0^2)^{62}.$$

\begin{itsapicture}
\begin{figure}[h!]
\begin{tikzpicture} [x=.5 cm,y=.5cm]
\draw (-1,0) -- (27.35,0) 
	(27.65,0) -- (31,0)
	(0,-.5) -- (0,4.5); 
\foreach \x in {1,5,25} 
	\draw (\x,-.1) node[below] {\x}-- (\x,.1);
\draw (30,-.1) node[below] {125}-- (30,.1)
	(27.35,-.3) -- (27.45,.3)
	(27.55,-.3) -- (27.65,.3);
\foreach \y in {1,...,4} 
	\draw (-.2,\y) node[left] {\y} -- (.2,\y);
\filldraw[blue, line width = 2pt] (30,0) circle (2pt) 
	(0,4) circle (2pt);
\node[below left] at (0,0) {0}; 
\draw[blue,line width =2pt] (0,4) -- (1,3) -- (5,2) -- (25,1) -- (27.3,.9) (28,.1) -- (30,0); 
\draw[dashed,line width = 1.5pt,blue] ; 
\foreach \x / \y in {1/3, 5/2, 25/1}
	{
	\draw[dashed,gray!60] (0,\y) -- (\x,\y) -- (\x,0);
	\foreach \i in {1,...,\x} 
		\filldraw[orange,line width = 2pt] (\i,\y) circle (2pt);
	}
\end{tikzpicture}
\caption{The $\phi$-Newton polygon for $T_5^3(x)-t_0$ with $\phi(x) = x+1$. The polynomial $\phi$ is determined by reducing $T_5^3(x)-t_0$ modulo 5: $T_5^3(x)-t_0 \equiv (x+1)^{125} \mmod 5$.}
\end{figure}
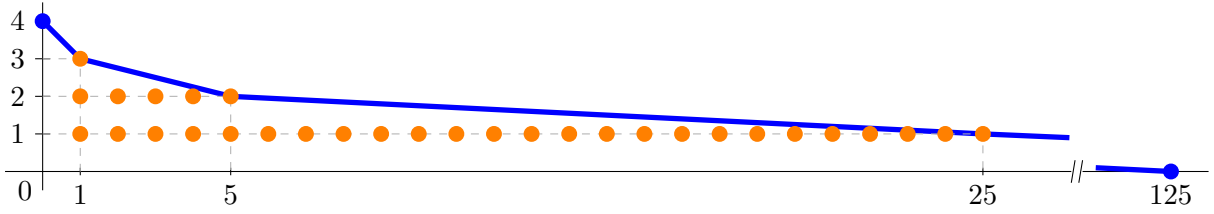
\end{itsapicture}

\item Consider the polynomial $T_7^3(x)-t_0$. Note that $\nu_7(T_7^3(t_0) - t_0) = 2$. By Corollary \ref{cor:5.9},
$$\ind_7(T_7^3(x)-t_0) = \sum_{i=1}^1 7^{3-i} = 49,$$
and by Corollary \ref{cor:fielddisctpm2sf},
$$\Delta(K_{7,3,t_0}) = 7^{931}(4-t_0^2)^{171}.$$
\end{enumerate}

\begin{itsapicture}
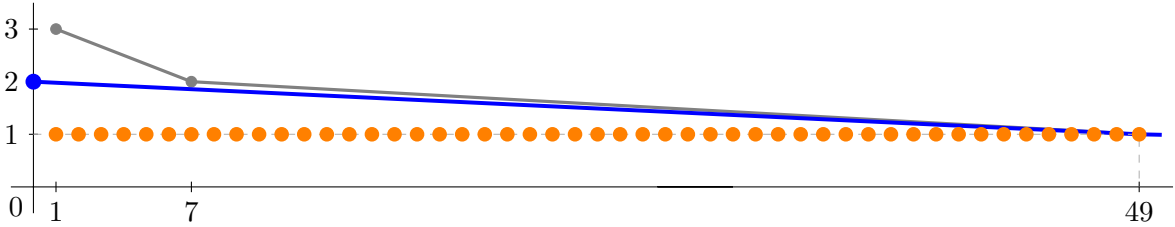
\begin{figure}[h!]
\begin{tikzpicture} [x=.3 cm,y=.7cm]
\draw (-1,0) -- (50.5,0) 
	(27.65,0) -- (31,0)
	(0,-.5) -- (0,3.5); 
\foreach \x in {1,7,49} 
	\draw (\x,-.1) node[below] {\x}-- (\x,.1);
\foreach \y in {1,...,3} 
	\draw (-.2,\y) node[left] {\y} -- (.2,\y);
\filldraw[blue, line width = 2pt] (0,2) circle (2pt); 
\filldraw[gray] (1,3) circle (2pt) 
	(7,2) circle (2pt);
\draw[very thick,gray] (1,3) -- (7,2) -- (49,1); 
\node[below left] at (0,0) {0}; 
\draw[blue,line width =1.5pt] (0,2) -- (49,1) -- (50,.99); 
\foreach \x / \y in {49/1}
	{
	\draw[dashed,gray!60] (0,\y) -- (\x,\y) -- (\x,0);
	\foreach \i in {1,...,\x} 
		\filldraw[orange,line width = 1.5pt] (\i,\y) circle (2pt);
	}
\end{tikzpicture}
\caption{The $\phi$-Newton polygon of $T_7^3(x)-t_0$ where $\phi(x) = x+3$. The gray line indicates the $\phi$-Newton polygon of $T_7^3(x)-t_0$ before considering the constant term of the $\phi$-development. The end vertex of the polygon (343,0) is not shown.}
\end{figure}
\end{itsapicture}
\end{example}

\section{Index calculations: On the multiplicity of $p$} \label{sec:p-index}
In the proof of Theorem \ref{th:dc1}, we showed that for primes $p \neq \ell$, $\bbz[\theta]$ is $p$-maximal if and only if $t \not\equiv \pm2 \mmod{p^2}$. In this section, we give a formula for $\ind_p(\Phi)$ when $p$ is an odd prime different from $\ell$ and $t$ is odd. Under these conditions, $p$-regularity is not guaranteed, however, it turns out that regularity is not a necessary condition in all cases. Let $t \equiv \pm2 \mmod{p^2}$ and write
\begin{align} \label{eq:7.1}
\Phi(x) \equiv (x\pm2)\phi_1(x)^2 \cdots \phi_r(x)^2\pmod p,
\end{align}
where $\phi_i(x)$ are irreducible factors modulo $p$. We prove the following.

\begin{thm} \label{th:6.1}
Let $p \neq \ell$ be an odd prime such that $t \equiv \pm2 \mmod{p^2}$. Then for each irreducible factor $\phi_i$ in equation \eqref{eq:7.1}, there exists a polynomial $\hat\phi_i$ such that $\hat\phi_i \equiv \phi_i\mmod p$, and the $\hat\phi_i$-polynomial is one-sided with vertices $(0,\nu_p(t^2-4))$ and $(2,0)$, i.e.
$$\ind_{\hat\phi_i}(\Phi) = \Floor{\frac{\nu_p(t^2-4)}{2}}\deg(\hat\phi_i).$$
Moreover, the factor $x\pm2$ does not contribute to $\ind_p(\Phi)$, i.e. $\ind_{(x\pm2)}(\Phi) = 0$.
\end{thm}

Consequently, if $\nu_p(t^2-4)$ is odd, then the residual polynomial associated with the $\hat\phi_i$-polygon is degree 1, and hence $\Phi$ is $\hat\phi_i$-regular for each $\hat\phi_i$. It follows from Theorem \ref{th:index} that
$$\ind_p(\Phi) = \sum_{i=1}^r \Floor{\frac{\nu_p(t^2-4)}{2}}\deg(\hat\phi_i) = \Floor{\frac{\nu_p(t^2-4)}{2}}\frac{\ell^n-1}{2}.$$
If $\nu_p(t^2-4)$ is even, regularity is not guaranteed since the residual polynomial is degree 2, so at best
$$\ind_p(\Phi) \ge \sum_{i=1}^r \Floor{\frac{\nu_p(t^2-4)}{2}}\deg(\hat\phi_i) = \frac{\nu_p(t^2-4)}{2}\frac{\ell^n-1}{2}.$$
On the other hand, the valuation of the index is bounded by the valuation of $\disc(\Phi)$.
$$\ind_p(\Phi) \le \frac{1}{2}\nu_p\left((t^2-4)^{(\ell^n-1)/2}\right) = \frac{\nu_p(t^2-4)}{2}\frac{\ell^n-1}{2}.$$
Thus we have derived the following result.

\begin{cor}
If $p\neq\ell$ is an odd prime and $t \equiv \pm2 \mmod{p^2}$, then 
$$\ind_p(\Phi) = \Floor{\frac{\nu_p(t^2-4)}{2}}\frac{\ell^n-1}{2}.$$
\end{cor}

Now that the complete factorization of $\ind(\Phi)$ has been determined, the following conditional discriminant formula follows from equation \eqref{eq:1.3}.

\begin{cor} \label{cor:6.3}
Let $\Phi$ and $K$ as before, and write $t^2-4 = A^2B$ where $B$ is square-free. If $t \not\equiv \pm2 \mmod{\ell^2}$ and $t \not \equiv 2 \mmod 4$, then
\begin{align*}
\Delta(K) = \ell^{n\ell^n-2\cdot\tiny\ind_\ell(\Phi)}\prod_{p \text{ \rm\tiny prime\,} \mid B} p^{(\ell^n-1)/2},
\end{align*}
where $\ind_\ell(\Phi)$ is defined in Corollary \ref{cor:5.9}.
\end{cor}

We conclude this section with the proof of Theorem \ref{th:6.1}.

\begin{proof}(Theorem \ref{th:6.1})
Recall that $T_\ell^n(x)\pm2 = (x\pm2)\tau(x)^2$ where 
\begin{align*}
\tau(x) \equiv \phi_1(x) \cdots \phi_r(x)\pmod p.
\end{align*}
Since $\tau$ has no repeated roots modulo $p$, Hensel lifting ensures that there exist lifts $\hat\phi_1, \ldots, \hat\phi_r$ such that 
\begin{align*}
\tau(x) \equiv \hat\phi_1(x) \cdots \hat\phi_r(x) \pmod{p^e}
\end{align*}
for $e$ arbitrarily large. Take $e > \nu_p(t^2-4)$ (although $e > \nu_p(t^2-4)/2$ would be sufficient) and fix a lift $\phi = \hat\phi_i$. Then the $\phi$-development of $T_\ell^n(x)\pm 2$ is 
\begin{align*}
T_\ell^n(x)\pm2 = A_0(x) + A_1(x)\phi(x) + A_2(x)\phi(x)^2 + \cdots.
\end{align*}
Note that $T_\ell^n(x)\pm2 = (x\pm2)\tau(x)^2 \equiv (x\pm2)\hat\phi_1(x)^2 \cdots \hat\phi_r(x)^2 \mmod{p^e}$, and hence $\nu_p(A_2) = 0$ and
\begin{align*}
A_0(x) + A_1(x)\phi(x) \equiv 0 \pmod{p^e}.
\end{align*}
In particular, since $\phi$ is monic, $\nu_p(A_0) \ge \nu_p(A_1) \ge e > \nu_p(t^2-4)$. Thus the $\phi$-development of $\Phi$ is
\begin{align*}
\Phi(x) &= T_\ell^n(x)-t = T_\ell^n(x)-\overline t + \overline t - t \\
&= \overline t - t + A_0(x) + A_1(x)\phi(x) + A_2(x)\phi(x)^2 + \cdots
\end{align*}
where $\nu_p(\overline t - t + A_0) = \nu_p(\overline t - t) = \nu_p(t^2-4), \nu_p(A_1) > \nu_p(t^2-4)$, and $\nu_p(A_2) = 0$, and therefore $\hat\phi_1, \ldots, \hat\phi_r$ provide desired lifts.

We now show that $\ind_{(x\pm2)}(\Phi) = 0$. The $(x\pm2)$-development is given by Taylor's expansion centered at $\pm2$:
\begin{align*}
\Phi(x) &= \Phi(\pm2) + \Phi'(\pm2)(x\pm2) + \cdots\\
&= \Phi(\pm2) + \ell^nU_{\ell^n-1}(\pm2)(x\pm2) + \cdots,
\end{align*}
where $U_d$ denotes the degree-$d$ Chebyshev polynomial of the second kind. Recalling the recursion $U_d(x) = xU_{d-1}(x) -U_{d-2}(x)$ (and $U_0(x) = 1, U_1(x) = x$), it is straightforward induction to show that $U_d(2) = d+1$. Moreover, since $U_{\ell^n-1}$ is an even function, it follows that $\nu_p(\ell^nU_{\ell^n-1}(\pm2)) =\nu_p(\ell^{2n})= 0$, and thus the $(x\pm2)$-polygon is one-sided with vertices $(0,\nu_p(\Phi(\pm2)))$ and $(1,0)$. We conclude that $\ind_{(x\pm2)}(\Phi) = 0$.
\end{proof}

\section{Integral basis} \label{sec:integral basis}
The Montes algorithm also provides an efficient method for determining an integral basis for the ring of integers $\cal O_K$. In this section we summarize their procedure as it pertains to our situation.

For this discussion we assume that $\Phi$ is regular with respect to every prime. Fix a prime $p$ for which $\bbz[\theta]$ is not maximal. Let $\hat\phi_i$ be a lift of an irreducible factor of $\overline\Phi$ for which $\Phi$ is $\hat\phi_i$-regular. We define the quotients attached to the $\hat\phi_i$-developement of $\Phi$ to be the polynomials
\begin{align*}
\Phi(x) &= \hat\phi_i(x)q_{i,1}(x) + a_{i,0}(x)\\
q_{i,1}(x) &= \hat\phi_i(x) q_{i,2}(x) + a_{i,1}(x)\\
&\;\;\vdots\\
q_{i,r-1}(x) &= \hat\phi_i(x)q_{i,r}(x) + a_{i,r-1}(x)\\
q_{i,r}(x) &= a_{i,r}(x).
\end{align*}
Additionally, for $1 \le j \le r$, we identify the points $(j,y_{i,j})$ on the polygon $N_{\hat\phi_i}(\Phi)$. 

\begin{cor} \label{th:ellintegralbasis}
The collection $\{q_{i,j}(\theta)/p^{\floor{ y_{i,j}}}\}$ contains a $p$-integral basis for $\cal O_K$.
\end{cor}

\begin{proof}
This is a specialization of \cite{EMN09}, Theorem 2.6.
\end{proof}

In section \ref{sec:ell-index}, we precisely determined the $\phi$-polygon for $\Phi$ for certain values of $t$. Under these same conditions, we determine generators for the ring $\cal O_K$.

\begin{prop} \label{prop:7.2}
Suppose that $t-2$ and $t+2$ are square-free, $\Phi(t) \equiv 0 \mmod{\ell^2}$, and $t \not\equiv \pm2 \mmod{\ell^2}$. Let $v = \min\{\nu_\ell(\Phi(t))-1,n\}$. Then
$$\cal O_K = \bbz\left[\theta, \frac{q_{\ell^{n-1}}(\theta)}{\ell}, \frac{q_{\ell^{n-2}}(\theta)}{\ell^2}, \ldots, \frac{q_{\ell^{n-v}}(\theta)}{\ell^v}\right].$$
\end{prop}

\begin{proof}
Recall that $\Phi(x) = T_\ell^n(x)-t \equiv (x-t)^{\ell^n} \mmod \ell$, so let $\phi(x) = x-\overline t$. In Corollary \ref{cor:5.9} we determined $N_\phi(\Phi)$ and showed that $\Phi$ is $\ell$-regular. For each $1 \le j \le \ell^n$, the quotient $q_j(x)$ is a monic polynomial of degree $\ell^n-j$, and these quotients satisfy the recursion $q_j(x) = \phi(x) q_{j+1}(x) + a_j$ where $q_{\ell^n}(x) = 1$. By definition, $\nu_\ell(a_j) \ge \floor{y_j}$. Hence if $\floor{y_{j+1}} = \floor{y_j}$, then $q_{j+1}(\theta)/\ell^{\floor{y_{j+1}}} \in\cal O_K$ implies that $q_j(\theta)/\ell^{\floor{y_j}} \in \cal O_K$. It follows that 
$$\cal O_K = \bbz\left[\frac{q_{\ell^n}(\theta)}{\ell^{\floor{y_{\ell^n}}}}, \ldots, \frac{q_1(\theta)}{\ell^{\floor{y_1}}}\right] = \bbz\left[\theta, \frac{q_{\ell^{n-1}}(\theta)}{\ell}, \frac{q_{\ell^{n-2}}(\theta)}{\ell^2}, \ldots, \frac{q_{\ell^{n-v}}(\theta)}{\ell^v}\right].$$
\end{proof}

\section{Acknowledgements}
The author would like to thank Farshid Hajir for his helpful conversations, continued support, and guidance. Additionally, the author extends his thanks to Jeff Hatley and Nico Aiello for their thought-provoking discussions.

\bibliographystyle{plain}
\bibliography{DiscOfChebyRadicalExt}

\begin{thebibliography}{10}

\bibitem{AHM05}
W.~Aitken, F.~Hajir, and C.~Maire.
\newblock Finitely ramified iterated extensions.
\newblock {\em IMRN}, 14:855--880, 2005.

\bibitem{Cohen95}
H.~Cohen.
\newblock {\em A Course in Computational Algebraic Number Theory}.
\newblock Graduate Texts in Mathematics. Springer, 1995.

\bibitem{EMN09}
L.~el~Fadil, J.~Montes, and E.~Nart.
\newblock Newton polygons and $p$-integral bases, 2009.
\newblock arXiv:0906.2629v1 [math.NT].

\bibitem{Gassert12}
T.~A. Gassert.
\newblock Chebyshev action on finite fields, 2012.
\newblock arXiv:1209.4396.v2 [math.NT].

\bibitem{GMN08}
J.~Gu\`ardia, J.~Montes, and E.~Nart.
\newblock Higher newton polygons in the computation of discriminants and prime
  ideal decomposition in number fields.
\newblock {\em J. Th\'eo. Nombres Bordeaux}, 23(3):667--696, 2011.

\bibitem{GMN09}
J.~Gu\`ardia, J.~Montes, and E.~Nart.
\newblock Higher newton polygons and integral bases, 2012.
\newblock arXiv:0902.3428v3 [math.NT].

\bibitem{GMN12}
J.~Gu\`ardia, J.~Montes, and E.~Nart.
\newblock Newton polygons of higher order in algebraic number theory.
\newblock {\em Trans. Amer. Math Soc.}, 364(1):361--416, 2012.

\bibitem{Kummer52}
E.~Kummer.
\newblock \"{U}ber die erg\"anzungss\"atze zu den allgemeinen
  reziprokcit\"ausgesetzten.
\newblock {\em Journal f\"ur die reine und angewandte Mathematik}, 44:93--146,
  1852.

\bibitem{Lucas78}
E.~Lucas.
\newblock Sur les congruences des nombres eul\'eriens et les coefficients
  diff\'erentiels des functions trigonom\'etriqu\'es suivant un module premier.
\newblock {\em Bulletin de la Soci\'et\'e Math\'ematique de France}, 6:49--54,
  1878.

\bibitem{Narkiewicz04}
W.~Narkiewicz.
\newblock {\em Elementary and Analytic Theory of Algebraic Numbers}.
\newblock Springer Monographs in Mathematics. Springer, 2004.

\bibitem{Rivlin90}
T.~J. Rivlin.
\newblock {\em Chebyshev Polynomials: From Approximation Theory to Algebra and
  Number Theory}.
\newblock Pure and Applied Mathematics. Wiley, 1990.

\bibitem{Silverman07}
J.~Silverman.
\newblock {\em The Arithmetic of Dynamical Systems}.
\newblock Graduate Texts in Mathematics. Springer, 2007.

\end{thebibliography}
\end{document}